\documentclass[12pt]{article}
\usepackage{amsmath}
\usepackage{amssymb}
\usepackage{amsthm}
\usepackage{pstricks}
\usepackage{graphics,color,url}
\usepackage{epsf}
\usepackage{hyperref}
\newtheorem{prop}{Proposition}[section]
\newtheorem{thm}[prop]{Theorem}
\newtheorem{lem}[prop]{Lemma}

\newcommand{\ra}{\rightarrow}

\begin{document}

\begin{center}
\vskip 1cm{\LARGE\bf On Packing Densities of Set Partitions}

\vskip 1cm
\large
Adam M.Goyt\footnote{Corresponding Author.  Phone 218.477.2206}\\
Department of Mathematics\\
Minnesota State University Moorhead\\
Moorhead, MN 56563, USA \\
\href{mailto:goytadam@mnstate.edu}{\tt goytadam@mnstate.edu}\\
\ \\
Lara K. Pudwell \\
Department of Mathematics and Computer Science\\
Valparaiso University\\
Valparaiso, IN 46383, USA\\
\href{mailto:Lara.Pudwell@valpo.edu}{\tt Lara.Pudwell@valpo.edu}\\
\end{center}

\begin{abstract}
We study packing densities for set partitions, which is a generalization of packing words.  We use results from the literature about packing densities for permutations and words to provide packing densities for set partitions.  These results give us most of the packing densities for partitions of the set $\{1,2,3\}$.  In the final section we determine the packing density of the set partition $\{\{1,3\},\{2\}\}$.

\bigskip\noindent \textbf{Keywords:} packing density, set partitions, words  
\end{abstract}

\section{Introduction}

Pattern avoidance and containment in combinatorial objects have been studied since they were introduced by Knuth~\cite{Knuthvol3}.  The first systematic study of pattern avoidance in permutations was done by Simion and Schmidt~\cite{SimionSchmidt}.  Burstein~\cite{B98} introduced pattern avoidance in words.  Klazar~\cite{abbafree,KlazarPartI,KlazarPartII} and Sagan~\cite{Saganpartitionpatterns} introduced the idea of pattern avoidance in set partitions.  In this paper we will explore the idea of packing patterns into set partitions.  That is to say, instead of trying to avoid a particular pattern we will find set partitions with the most copies of a pattern.  We will use this information to describe packing densities for different patterns.  

The idea of packing permutations was first studied by Stromquist~\cite{StromquistUnpub} in an unpublished paper and carried on by Price~\cite{Pricediss} in his dissertation.  Many people~\cite{AAHHS02,Hasto02,HSV04,Warren04,Warren06} advanced the study of packing permutations, and Burstein, H\"{a}st\"{o} and Mansour~\cite{BHM03} extended the concept of packing to words.  This paper is the first attempt at packing set partitions.  We will see that this is closely related to packing words, and depending on the definition of pattern containment in set partitions, some of the results on words carry over to this new context.  We begin with some definitions.   

Let $[n]=\{1,2,\dots,n\}$.  A {\it partition} $\pi$ of $[n]$ is a family of disjoint sets $B_1$, $B_2$, $\dots$, $B_k$ called {\it blocks} such that $\bigcup_{i=1}^k B_i=[n]$. We write $\pi=B_1/B_2/\dots/B_k$ where $$\min B_1<\min B_2<\dots<\min B_k.$$ For example $\pi=145/26/37$ is a partition of the set $[7]$.  Notice that $\pi$ has three blocks.  Let $\Pi_n$ be the set of partitions of $[n]$ and $\Pi_{n,k}$ be the set of partitions of $[n]$ with at most $k$ blocks.    

Let $\pi=B_1/B_2/\dots/B_k$ be a partition of $[n]$.  We associate to $\pi$ the word $\pi_1\pi_2\cdots\pi_n$, where $\pi_i=j$ if and only if $i\in B_j$.  So the word associated to the partition $145/26/37$ is $1231123$.  

Let $[k]^n$ be the set of words with $n$ letters from the alphabet $[k]$.  If $w\in [k]^n$, we may {\it canonize} $w$ by replacing all occurrences of the first letter by 1, all occurrences of the second occurring letter by 2, etc.  For example the word $w=3471344574$ has canonical form $1234122532$.  The set $\Pi_n$ and the set of all canonized words of length $n$ are in obvious bijection with each other.

Let $u=u_1u_2\cdots u_n$ and $w=w_1w_2\cdots w_n$ be words.  We say that $u$ and $w$ are {\it order isomorphic} if $u_i = u_j$ respectively $u_i<u_j$ if and only if $w_i = w_j$ respectively $w_i<w_j$ for any $1\leq i \not= j \leq n$.    

For the duration of this paper we will discuss set partitions in the form of canonized words.  We say that a partition $\sigma=\sigma_1\sigma_2\cdots\sigma_n$ of $[n]$ {\it contains} a copy of partition $\pi=\pi_1\pi_2\cdots\pi_k$ of $[k$] in the {\it restricted sense} if there is a subsequence $\sigma'=\sigma_{i_1}\sigma_{i_2}\cdots\sigma_{i_k}$ such that such that $\sigma'$ and $\pi$ are order isomorphic.  We say that a partition $\sigma=\sigma_1\sigma_2\cdots\sigma_n$ of $[n]$ {\it contains} a copy of partition $\pi=\pi_1\pi_2\cdots\pi_k$ of $[k]$ in the {\it unrestricted sense} if there is a subsequence $\sigma'=\sigma_{i_1}\sigma_{i_2}\cdots\sigma_{i_k}$ such that the canonization of 
$\sigma'$ is $\pi$.  If a partition $\sigma$ does not contain a copy of $\pi$ in the (un)restricted sense then we say that $\sigma$ \emph{avoids} $\pi$ in the (un)restricted sense.  

For example the partition $1213221$ contains many copies of $121$.  Positions two, four and five give the subsequence $232$ which is a copy of 121 in the restricted sense and the unrestricted sense.  Positions two, three and five give the subsequence $212$ which is only a copy in the unrestricted sense.  Furthermore, this partition avoids $1112$ in the restricted sense, but not the unrestricted sense, since the sequence $2221$ canonizes to $1112$.  

Let $S\subset \Pi_m$ and let $\nu_r(S,\pi)$ (respectively $\nu(S,\pi)$) be the number of copies of partitions from $S$ in $\pi$ in the restricted (respectively unrestricted) sense.  Let $$\mu_r(S,n,k)=\max\{\nu_r(S,\pi):\pi\in\Pi_{n,k}\},$$ and $$\mu(S,n,k)=\max\{\nu(S,\pi):\pi\in\Pi_{n,k}\}.$$  The probability of a randomly chosen subsequence of a partition $\pi$ to be a partition from $S$ in the restricted sense is $$d_r(S,\pi)=\frac{\nu_r(S,\pi)}{\binom{n}{m}}$$ and in the unrestricted sense is $$d(S,\pi)=\frac{\nu(S,\pi)}{\binom{n}{m}}.$$  The maximum probability is  $$\delta_r(S,n,k)=\frac{\mu_r(S,n,k)}{\binom{n}{m}}$$ and  $$\delta(S,n,k)=\frac{\mu(S,n,k)}{\binom{n}{m}},$$ respectively.  

The restricted sense of pattern containment in set partitions is the traditional definition.  It is most closely related to the definition of pattern containment in permutations as defined by Knuth~\cite{Knuthvol3}.  As such, when Burstein~\cite{B98} took on the study of pattern containment and avoidance in words, he defined pattern containment in words as follows.  A word $w=w_1w_2\dots w_n\in [\ell]^n$ {\it contains} a word $u=u_1u_2\dots u_m\in[k]^m$ if there is a subword $w'=w_{i_1}w_{i_2}\dots w_{i_m}$ that is order isomorphic to $u$.  Otherwise we say that $w$ {\it avoids} $u$.  This is exactly the restricted containment definition for set partitions.  We simply focus on canonized words. 

For a set of patterns $S\subset [k]^n$, Burstein, H\"{a}st\"{o}, and Mansour~\cite{BHM03} define $\hat{\nu}(S,\sigma)$ to  be the number of occurrences of patterns from $S$ in $\sigma$, and $$\hat{\mu}(S,n,k)=\max\{\hat{\nu}(S,\sigma):\sigma\in[k]^n\},$$ $$\hat{d}(S,\sigma)=\frac{\hat{\nu}(S,\sigma)}{\binom{n}{m}},$$ and $$\hat{\delta}(S,n,k)=\frac{\hat{\mu}(S,n,k)}{\binom{n}{m}}=\max\{\hat{d}(S,\sigma):\sigma\in[k]^n\}.$$

\begin{prop} For a set $S\subset \Pi_m$ of set partition patterns, we have $$\delta_r(S,n,k)=\hat{\delta}(S,n,k).$$\end{prop}

\begin{proof}  Let $S\subset \Pi_m$.  It suffices to show that $\mu_r(S,n,k)=\hat{\mu}(S,n,k)$.  Since $\Pi_{n,k}\subset[k]^n$ we have that $\mu_r(S,n,k)\leq \hat{\mu}(S,n,k)$.  We need only show the opposite inequality.  

Let $\sigma\in [k]^n$ satisfy $\hat{\nu}(S,\sigma)=\hat{\mu}(S,n,k)$.  Rewrite $\sigma$ using the smallest alphabet possible by replacing the smallest element by 1, the next smallest by 2, etc.  Call this new word $\tilde{\sigma}$.  Let $\tilde{\sigma}=\sigma_1\sigma_2\cdots\sigma_n$.  If $\tilde{\sigma}\in\Pi_{n,k}$ then we are done. 

 If $\tilde{\sigma}\not\in \Pi_{n,k}$ then suppose that $i\in[n]$ is the first position such that $\sigma_1\cdots\sigma_{i-1}\in\Pi_{i-1,k}$ and $\sigma_i>\max\{\sigma_j:1\leq j\leq i-1\}+1$.  If $\sigma_1\not=1$ then in the following argument let $i=1$ and set $\max\{\sigma_j:1\leq j\leq i-1\}=0$.  Let $t\in[n]$ be the smallest element such that $\sigma_t=\max\{\sigma_j:1\leq j\leq i-1\}+1$.  Any copy of an element from $S$ that involves $\sigma_t$ cannot involve any of the elements $\sigma_i,\sigma_{i+1},\dots,\sigma_{t-1}$.  So we do not lose any copies of elements from $S$ if we move the element $\sigma_t$ into the $i^{th}$ position.  Now, the word $\sigma_1\cdots\sigma_{i-1}\sigma_t\in\Pi_{i,k}$.  By induction we can find a word $\bar{\sigma}\in\Pi_{n,k}$ such that $\nu(S,\bar{\sigma})=\hat{\mu}(S,n,k)$.  
 
Thus, $\mu_r(S,n,k)\geq \hat{\mu}(S,n,k)$, and hence $\mu_r(S,n,k)= \hat{\mu}(S,n,k)$.\end{proof}

We are interested in the asymptotic behavior of $\delta_r(S,n,k)$ and $\delta(S,n,k)$ as $n\rightarrow \infty$ and $k\rightarrow\infty$.  By work done by Burstein, H\"{a}st\"{o} and Mansour~\cite{BHM03} for $S\subset\Pi_m$ we have that $\delta_r(S,n,k)\leq\delta_r(S,n-1,k)$ and $\delta_r(S,n,k)\geq\delta_r(S,n,k-1)$.  They show further that $\lim_{n\rightarrow\infty}\lim_{k\rightarrow\infty}\delta_r(S,n,k)$ and $\lim_{k\rightarrow\infty}\lim_{n\rightarrow\infty}\delta_r(S,n,k)$ exist.  Let's define these limits to be $\delta_r(S)$ and $\delta_r'(S)$ respectively.  We will give a similar result for unrestricted patterns.  

\begin{prop} Let $S\subset \Pi_m$, then for $n>m$ we have $\delta(S,n-1,k)\geq\delta(S,n,k)$, $\delta(S,n,k)\geq\delta(S,n,k-1)$.\end{prop}

\begin{proof} The inequality $\delta(S,n-1,k)\geq\delta(S,n,k)$ follows from the proof of Proposition 1.1 in~\cite{AAHHS02}.  The repetition of letters is irrelevant, and we can simply canonize the resulting partition.

We have that $\delta(S,n,k)\geq\delta(S,n,k-1)$, since allowing for more blocks only increases the number of possible patterns. \end{proof}

Notice that a partition of $[n]$ can have at most $n$ blocks, so $\lim_{k\ra\infty} \delta(S,n,k)=\delta(S,n,n)$.  Furthermore, we have that $\delta(S,n,n)=\delta(S,n,n+1)\geq\delta(S,n+1,n+1)$.  Thus, $\{\delta(S,n,n)\}$ is nonnegative and decreasing and hence $$\delta(S)=\lim_{n\ra\infty}\lim_{k\ra\infty} \delta(S,n,k)$$ exists.  We call $\delta(S)$ the {\it packing density} of $S$.  

Of course we could take the limits in the opposite order.  That is consider the double limit $\lim_{k\ra\infty}\lim_{n\ra\infty}\delta(S,n,k)$.  Since $\delta(S,n,k)$ is decreasing in $n$ and nonnegative, we have that $\lim_{n\ra\infty}\delta(S,n,k)$ exists.  Now, $\lim_{n\ra\infty}\delta(S,n,k)$ is increasing in $k$ and bounded above by 1, thus we may define $$\delta'(S)=\lim_{k\ra\infty}\lim_{n\ra\infty}\delta(S,n,k).$$

An important question is whether $\delta'(S)=\delta(S)$. Burstein, H\"{a}st\"{o}, and Mansour~\cite{BHM03} conjectured that $\delta_r(S) = \delta_r'(S)$ and Barton~\cite{Barton04} proved it. It turns out that Barton's proof works for the unrestricted case as well.  

\begin{lem}[Barton] Let $S$ be a collection of patterns of length $m$ and $\sigma\in[n]^n$ be an $S$-maximizer.  Then $${n\choose m}\delta_r(S)\leq \nu_r(S,\sigma)\leq \frac{n^m}{m!}\delta_r'(S).$$\end{lem}
\begin{proof}  We know that $\delta_r(S,n,n)\geq \delta_r(S)$, so there is some $\sigma\in[n]^n$ satisfying $d_r(S,\sigma)\geq \delta_r(S)$, so $\nu_r(S,\sigma)={n\choose m}d_r(S,\sigma)\geq {n\choose m}\delta_r(S)$.  This gives the left-hand inequality.

To show the right-hand inequality we will show that given any $\sigma\in [n]^n$ we have that $\nu_r(S,\sigma)\leq\frac{n^m}{m!}\delta_r'(S)$.  For $t\geq1$, form the word $\sigma_t\in[n]^{tn}$ by repeating each letter of $\sigma$ $t$ times.  Now, every occurrence of a pattern $\pi\in S$ gives rise to $t^m$ occurrences of $\pi$ in $\sigma_t$, so $\nu_r(S,\sigma_t)\geq t^m\nu_r(S,\sigma)$.  Thus, 
$$\delta_r(S,n)=\lim_{t\rightarrow \infty}\delta_r(S,n,tn)\geq\lim_{t\rightarrow \infty}d_r(S,\sigma_t)\geq\lim_{t\rightarrow \infty}\frac{t^m\nu_r(S,\sigma)}{{tn\choose m}}=\frac{m!}{n^m}\nu_r(S,\sigma).$$  Now, $\delta_r'(S)\geq\delta_r(S,n)$, so the right hand inequality is proved. \end{proof}

The argument in Barton's proof holds whether we restrict the types of copies in a word or not.  Also, the construction of $\sigma_t$ from $\sigma$ will maintain the canonical form of the word.  So we could delete the subscript $r$ everywhere in the previous proof and lemma and have the same result.

\begin{lem} \label{Lemma:unrestricted} Let $S$ be a collection of patterns of length $m$ and $\sigma\in[n]^n$ be an $S$-maximizer.  Then $${n\choose m}\delta(S)\leq\nu(S,\sigma)\leq\frac{n^m}{m!}\delta'(S).$$\end{lem}

\begin{thm} Let $S\subset \Pi_m$.  Then $\delta(S)=\delta'(S)$.  \end{thm}

\begin{proof} We know from above that $\delta(S,k,k)\geq\delta(S,k)$ for $k\geq m$, so we have that $$\delta(S)=\lim_{k\rightarrow \infty}\delta(S,k,k)\geq\lim_{k\rightarrow \infty}\delta(S,k)=\delta'(S).$$ On the other hand, using Lemma \ref{Lemma:unrestricted} we have that ${n\choose m}\delta(S)\leq\frac{n^m}{m!}\delta'(S)$, so letting $n$ approach infinity gives us that $\delta(S)\leq \delta'(S)$. \end{proof}

Our main focus will be to determine $\delta(S)$ where $S\subset \Pi_3=\{111,112,121,122,123\}$ and $|S|=1$.  The patterns $112$ and $122$ are equivalent in the unrestricted sense because if $\sigma=\sigma_1\sigma_2\cdots\sigma_n$ contains $m$ copies of $112$ then the partition obtained by canonizing $\sigma'=\sigma_n\sigma_{n-1}\cdots\sigma_1$ contains $m$ copies of $122$.  Thus, we only need to determine the packing densities of each of the patterns $111$, $112$, $121$ and $123$.  

In the next section we will use previous results on words to answer questions about $\delta_r(S)$ for certain sets $S\subset \Pi_3$.  In Section 3 we will discuss some of the subtle differences between restricted and unrestricted copies and determine values of $\delta(S)$  for certain sets $S\subset \Pi_3$.  In Section 4 we will tackle the remaining partition of $\Pi_3$, the so called unlayered partition.  We will conclude by suggesting open problems.

\section{Packing in the Restricted Sense}

By Proposition 1.1, we have that $\delta_r(S,n,k)=\hat{\delta}(S,n,k)$.  This implies that the packing densities in the restricted sense are the same as the packing densities determined by Burstein, H\"{a}st\"{o} and Mansour~\cite{BHM03}.  We give their results here.  We give proofs for the first two and refer the reader to their paper for the remaining proofs.  

Consider the partition, $\beta_m$ of $[m]$ where every element is in the same block.  That is $\beta_m$ is a string of $m$ $1$'s.  In this case a copy of $c_m$ in a partition $\sigma$ is any constant sequence of length $m$.  Clearly, $d_r(\beta_m,\beta_n)=1$ for $n\geq m$, and hence $\delta_r(\beta_m)=1$ for any $m\geq1$.  

Now consider the opposite extreme $\gamma_m=12\cdots m$, i.e. the partition with every element in its own block.  Any copy of $\gamma_m$ is a strictly increasing sequence of length $m$.  Clearly, $d_r(\gamma_m,\gamma_n)=1$ for $n\geq m$, and hence $\delta_r(\gamma_m)=1$ for $m\geq 1$.

The packing densities in the restricted sense for the partitions of $[3]$ are given in the table below.  

\bigskip  

\begin{center}
\begin{tabular}{|c||c|c|c|c|} \hline
Partition $\pi$&111&112&121&123\\ \hline 
Packing Density $\delta_r(\pi)$&1&$2\sqrt{3}-3$&$\frac{2\sqrt{3}-3}{2}$&1\\ \hline
\end{tabular}
\end{center}

\section{Packing in the Unrestricted Sense}

As we mentioned before, our goal is to determine the packing densities of the partitions of $[3]$.  The packing densities of $112$ and $122$ are equivalent, so we need only consider the packing densities of $111$, $112$, $121$, and $123$.  The arguments that $\delta_r(111)=\delta_r(123)=1$ also show that $\delta(111)=\delta(123)=1$.  The pattern $112$ is a layered partition, which we will define below.  The partition $121$ is not layered, and in fact is the smallest nonlayered partition.  We will determine the packing density of $121$ in Section 4.  We now turn our attention to layered partitions in order to deal with $112$.  

Let $\pi$ be a partition of $[n]$.  We say that $\pi$ is {\it layered} if $\pi=11\cdots122\cdots2\cdots kk\cdots k$, where $k\in \mathbb{N}$.  Let $\pi$ be a partition of $[n]$.  The number of elements in the $i^{th}$ block, $B_i$, is the number of occurrences of $i$ in $\pi$.  We will say that $\pi$ is {\it monotone layered} if $\pi$ is layered and $|B_1|\leq|B_2|\leq\cdots\leq |B_k|$ or $|B_1|\geq|B_2|\geq\cdots\geq |B_k|$.  For example, $1112223$ is monotone layered, but $111233$ is layered but not monotone, and $122113$ is monotone but not layered.  

Let $\pi$ be a partition of $[n]$.  We say the {\it block structure} of $\pi$ is the multiset of block sizes of $\pi$. For example the block structure of $\pi=1121222333$ is $\{3,3,4\}$, so while monotonicity cares about the order of the sizes of the blocks, the specific block structure does not.    

\begin{lem} Let $\pi=11\cdots122\cdots2\in\Pi_m$ be a monotone increasing layered partition.  For each $\sigma\in\Pi_{n,2}$, let $\tilde{\sigma}\in\Pi_{n,2}$ be the unique monotone increasing layered partition with the same block structure as $\sigma$. We have that $\nu(\pi,\tilde{\sigma})\geq\nu(\pi,\sigma)$. \end{lem}

\begin{proof}  Let $\pi$ be as described above and consider any partition $\sigma\in\Pi_{n,2}$.  If $\sigma$ has only one block then it is already layered and we are done. 

Suppose that $\sigma$ has two blocks, one of size $b_1$ and the other of size $b_2$, and suppose $b_2\geq b_1$.  Without loss of generality suppose that there are $b_1$ ones and $b_2$ twos.  Suppose the pattern $\pi$ has $a_1$ ones and $a_2$ twos.  We have two cases.  If $a_1=a_2$ then the maximal number of copies of $\pi$ in $\sigma$ is $2{b_1\choose a_1}{b_2\choose a_2}$.  This comes from the fact that given any $a_1$ of the ones in $\sigma$ there are at most $2{b_2\choose a_2}$ copies of $\pi$ involving these $a_1$ ones.  This maximum is achieved by the partition with $b_1$ ones followed by $b_2$ twos.  

Now, suppose that $a_1<a_2$.  If $n=m$ then the partition with the most copies of $\pi$ and the same block structure as $\sigma$ is $\pi$ itself which contains one copy.  Any others contain zero copies.  

Now, suppose that $n>m$.  We induct on $n$.  Remove the last letter from $\sigma$ and call this new partition $\sigma'$.  By induction there is a monotone increasing layered partition with the same block structure as $\sigma'$ that has at least as many copies of $\pi$ as $\sigma'$.  Now replace the last letter, and adjust so that the block structure of this new partition is the same as the original block structure of $\sigma$.  Call this new partition $\tilde{\sigma}$.  

We know that the number of copies of $\pi$ in $\tilde{\sigma}$ that do not include the last letter is at least as many as the number of copies of $\pi$ in $\sigma$ that do not include the last letter.  

We turn our attention to the number of copies of $\pi$ that do include the last letter.  Either the last letter in $\sigma$ was a 1 or a 2.  In $\tilde{\sigma}$ the last letter is a 2.

Suppose that there are $a_1$ 1's in $\pi$ and $a_2$ 2's in $\pi$. Suppose there are $b_1$ 1's in $\sigma$ and $b_2$ 2's in $\sigma$, and without loss of generality, assume that $b_2\geq b_1$.  There are $\binom{b_1}{a_1}\binom{b_2-1}{a_2-1}$ copies of $\pi$ in $\tilde{\sigma}$ that include the last letter of $\tilde{\sigma}$.  If the last letter in $\sigma$ was a 2 then there were at most $\binom{b_1}{a_1}\binom{b_2-1}{a_2-1}$ copies of $\pi$ involving $n$ in $\sigma$, which is the same as the number of such copies in $\tilde{\sigma}$.  If the last letter in $\sigma$ was a 1 then there were at most $\binom{b_1-1}{a_2-1}\binom{b_2}{a_1}$ copies of $\pi$ in $\sigma$ that involve the last letter, which is no more than the number of such copies of $\pi$ in $\tilde{\sigma}$.  That is to say, ${b_1-1\choose a_2-1}{b_2\choose a_1} \leq {b_1\choose a_1}{b_2-1\choose a_2-1}$.  The preceding inequality is inductively true, assuming that $a_1<a_2$.   \end{proof}

\begin{thm}  Let $\pi$ be a layered monotone increasing partition with exactly $k$ blocks.  For each $\sigma\in \Pi_n$, the layered monotone increasing partition, $\tilde{\sigma}$, with the same block structure as $\sigma$ satisfies $\nu(\pi,\tilde{\sigma})\geq\nu(\pi,\sigma)$.\end{thm}

\begin{proof}  Let $\pi$ be as described above, and assume that $\pi$ has exactly $k$ blocks.  Let $\sigma\in\Pi_n$, and assume that $\sigma$ has exactly $\ell$ blocks.  

Remove the last letter from $\sigma$, and call this new partition $\sigma'$.  By induction the layered monotone increasing partition $\tilde{\sigma'}$ with the same block structure as $\sigma'$ contains at least as many copies of $\pi$ as $\sigma'$.  

Now, replace the last letter and adjust so that the new partition, $\tilde{\sigma}$, has the same block structure as $\sigma$.  By the previous paragraph, we know that the number of copies of $\pi$ in $\tilde{\sigma}$ that do not involve the last letter is at least as many as the number of copies of $\pi$ in $\sigma$ that do not involve the last letter.  

We turn our attention to the number of copies that do involve the last letter.  Let $\nu(\pi,\sigma,n)$ be the number of copies of $\pi$ in $\sigma$ involving the last letter of $\sigma$.  Assume that the last letter in $\sigma$ is $j$.  Any copy of $\pi$ in $\sigma$ that involves the last letter, must have the $k$'s in $\pi$ corresponding to the $j$'s in $\sigma$.  Thus, we will not lose any copies of $\pi$ that involve the last letter by moving all of the $j$'s to the end of $\sigma$.  For ease of explanation, we will not canonize this new partition, and we will continue to call it $\sigma$.   

Let $\bar{\sigma}$ be the partition consisting of all but the $j$'s in $\sigma$, and let $\bar{\pi}$ be the partition consisting of the first $k-1$ blocks of $\pi$.  By induction on the number of blocks the number of copies of $\bar{\pi}$ in the layered monotone increasing partition, $\tilde{\bar{\sigma}}$, with the same block structure as $\bar{\sigma}$ is at least as many as the number of copies of $\bar{\pi}$ in $\bar{\sigma}$.  Note that we can obtain $\tilde{\bar{\sigma}}$ by moving elements around and canonizing using the elements $[1,j-1]\cup[j+1,\ell]$.  

Replace the first $\ell-1$ blocks of $\sigma$ by $\tilde{\bar{\sigma}}$, and call this new partition $\hat{\sigma}$.  We have that $\hat{\sigma}$ must be layered, but may or may not be monotone increasing.  Suppose that there are $b_j$ $j$'s in $\sigma$ and assume there are $b_\ell$ $\ell$'s in $\tilde{\sigma}$.  If $b_j=b_\ell$ then we are done.  If $b_j<b_\ell$, then by Lemma 3.1 we have $\nu(\pi,\hat{\sigma},n)\leq\nu(\pi,\tilde{\sigma},n)$.  By construction $\nu(\pi,\sigma,n)\leq\nu(\pi,\hat{\sigma},n)$.  

Thus, we have not reduced the number of copies of $\pi$ by replacing $\sigma$ by $\tilde{\sigma}$.  \end{proof}

Theorem 3.2 tells us that if $\pi$ is layered, monotone increasing, then if we want to know $\mu(\pi,n,k)$ we need only look at layered monotone increasing $\sigma\in\Pi_{n,k}$.  Of course everything we did in Lemma 3.1 and Theorem 3.2 can be done for layered monotone decreasing partitions.  This coincides with results of Burstein, H\"{a}st\"{o}, and Mansour~\cite{BHM03} on words and Price~\cite{Pricediss}, Albert, Atkinson, Handley, Holton, and Stromquist~\cite{AAHHS02} and Barton~\cite{Barton04} on permutations.  

Let a {\it nondecreasing layered word} be a word of the form $11\cdots122\cdots2\cdots kk\cdots k$, as defined in ~\cite{BHM03}. These are identical to layered partitions.  Furthermore, if $\pi$ and $\sigma$ are layered monotone increasing (decreasing) partitions then $\nu(\pi,\sigma)=\nu_r(\pi,\sigma)$.  Thus, we can use the results of~\cite{AAHHS02,Barton04,BHM03} to determine $\delta(\pi)$ where $\pi$ is a layered monotone increasing (decreasing) partition. 

The results of Price~\cite{Pricediss} give us that $\delta(112)=2\sqrt{3}-3$, $\delta(1122)=3/8$.  For $k\geq2$, $\delta(\underbrace{1\cdots 1}_{k}2)=k\alpha(1-\alpha)^{k-1}$, where $0<\alpha<1$ and $k\alpha^{k+1}-(k+1)\alpha+1=0$.  Furthermore, for $a,b\geq2$, $$\delta(\underbrace{1\cdots1}_{a}\underbrace{2\cdots2}_{b})=\binom{a+b}{a}\frac{a^ab^b}{(a+b)^{a+b}}.$$  The results of Albert et al.~\cite{AAHHS02} give us that $\delta(1123)=\delta(1233)=3/8$.  

\section{Packing 121}

In order to complete the determination of the packing densities of the partitions of $[3]$ we need to address the pattern 121.  We will prove that the partition of $[n]$ consisting of alternating 1's and 2's, i.e. $121212\cdots12$ is the maximizer.  

\begin{lem}  Let $\pi\in\Pi_{n,2}$ have exactly two blocks.  Assume that of the first $a+b$ elements $a$ are 1's and $b$ are 2's, and of the last $c+d$ elements $c$ are 1's and $d$ are 2's, where $n=a+b+c+d+2$.  If the $a+b+1^{st}$ element is a 2 and the $a+b+2^{nd}$ element is a 1 then switching the order of these two elements changes the number of copies of $121$ by $(b+c)-(a+d)$. \end{lem}

\begin{proof}  We have partition $\pi=\underbrace{\underline{\phantom{aaaaaaaaaaaaa}}}_{a\mbox{ }1's,\mbox{ }b\mbox{ }2's}21\underbrace{\underline{\phantom{aaaaaaaaaaaaa}}}_{c\mbox{ }1's,\mbox{ }d\mbox{ }2's}$.  By switching the 1 and 2 in positions $a+b+1$ and $a+b+2$, we obtain $\hat{\pi}=\underbrace{\underline{\phantom{aaaaaaaaaaaaa}}}_{a\mbox{ }1's,\mbox{ }b\mbox{ }2's}12\underbrace{\underline{\phantom{aaaaaaaaaaaaa}}}_{c\mbox{ }1's,\mbox{ }d\mbox{ }2's}$.  

The only copies of 121 that are lost or created are copies that involve both of these positions.  Thus, we lose $a$ copies of the form 121 and $d$ copies of the form 212.  We create $b$ copies of the form 212 and $c$ copies of the form 121.  This gives us a net change of $(b+c)-(a+d)$ copies.  \end{proof}

\begin{lem}  Let $\pi\in\Pi_{n,2}$ have exactly two blocks.  Assume that $\pi$ consists of $i$ 1's and $j$ 2's with $i\geq j$.  Then the partition $$\hat{\pi}=\underbrace{11\cdots1}_{\left\lceil(i-j-1)/2\right\rceil}\underbrace{1212\cdots121}_{2j+1}\underbrace{11\cdots1}_{\left\lfloor(i-j-1)/2\right\rfloor}$$ satisfies $\nu(121,\hat{\pi})\geq\nu(121,\pi)$.\end{lem}

\begin{proof} We begin by showing that the middle section of $\hat{\pi}$ must have this alternating format.  Suppose in $\pi$ there is a string of $\ell+2$ elements with $\ell\geq2$ where the first and last elements are 2's and the remaining $\ell$ elements are 1's.  Now suppose that preceding the first 2 are $a$ 1's and $b$ 2's and succeeding the last 2 are $c$ 1's and $d$ 2's.  If we swap the 2 immediately preceding this run of $\ell$ 1's with the first 1 in the run , we will have a change of $(b+c+\ell)-(a+d+2)$ copies of 121.  Swapping the last 1 in the run with the 2 immediately following it gives us a change of $(a+d+\ell)-(b+c+2)$ copies of 121.  Since $\ell\geq2$, at least one of these must be nonnegative, so we can perform one of these swaps without decreasing the number of copies of 121.  A similar argument holds if we replace the 2's by 1's and vice versa. This gives us that we must have alternating 1's and 2's in the middle of $\hat{\pi}$.  

We turn our attention to the number of $1$'s that precede and succeed this alternating run.  Suppose that the alternating section is as described in the statement of the lemma and is preceded by $a$ 1's and succeeded by $b$ 1's.  The number of copies of 121 that involve these outside 1's is given by $$\left(\sum_{k=1}^jka\right)+\left(\sum_{k=1}^jkb\right)+abj.$$  The first sum gives the number of copies of 121 involving the one of the first $a$ 1's and a pair from the alternating section.  The second sum gives the number of copies of 121 involving one of the last $b$ 1's and a pair from the alternating section.  The last term is the number of copies of 121 using a 1 from the first $a$ and a 1 from the last $b$ and a 2 from the alternating section.  This expression simplifies to $a\binom{j+1}{2}+b\binom{j+1}{2}+abj$ which is maximized when $a=b$.      \end{proof}

These first two lemmas tell us that if $\sigma\in\Pi_{n,2}$ then among all partitions with the same block structure as $\sigma$ the one with the structure described in Lemma 4.2 has the most copies of $121$.  Furthermore, among those with the structure described in Lemma 4.2, the one that consists entirely of an alternating section has the most copies of 121.

\begin{lem}  Suppose that $\pi\in\Pi_n$ has structure described in Lemma 4.2 with $a$ $1$'s at the beginning, an alternating section involving $j$ 2's and $j+1$ 1's, and $a$ or $a-1$ 1's at the end.  (If $a=0$ and $n$ is even then we allow the alternating section to end in a 2.)  Then the number of copies of $121$ is maximized when $a=0$.  \end{lem}

\begin{proof}  We begin with a partition $\pi$ that has the structure described above, and we assume that $a\geq 1$.  Since $a\geq1$ there is at least one extra 1 at the beginning and at least zero extra 1's at the end.  Assume that there are $a$ 1's at the beginning and the end.  By changing the last of the string of $a$ 1's at the beginning to a 2 and the first of the string of $a$ 1's at the end to a 2 we lose $2ja-j+2\binom{j+1}{2}$ copies of $121$ and gain $2(a-1)(j+a)+\binom{j+1}{2}+\binom{j+2}{2}$ copies of 121.  The net gain is $a^2+(a-1)^2$  copies of $121$.  

In the case where $\pi$ begins with $a$ 1's, ends in $(a-1)$ 1's and $a\geq2$, switching the last 1 in the first run to a 2 and the first 1 in the last run to a 2 gives a net gain of $2(a-1)^2$ copies of 121.  

Finally, in the case where $a=1$ and the last run of 1's consists of zero 1's we have two cases: either the alternating section ends in 1 or 2.  In this case we turn the first 1 into a 2.  If the alternating section ends in 1 then there is no net gain or loss of copies of 121.  If the alternating section ends in 2 there is a net gain of $j$ copies of 121.  In either of these cases we canonize after changing the 1 to a 2, to change the new word into a partition.  

Thus, the number of copies of $121$ in this case is maximized when $a=0$.  \end{proof}

Lemma 4.3 tells us that $\nu(121,\pi)$ for $\pi\in\Pi_{n,2}$ is maximized when $\pi$ is the partition consisting of alternating 1's and 2's.  We will now show that among partitions with any number of blocks the number of copies of $121$ is maximized by the partition consisting of alternating 1's and 2's.  We call the alternating partition of length $n$ $\alpha_n$.  Notice that $\nu(121,\alpha_n)=\frac{1}{24}(n^3-n)$ if $n$ is odd and $\nu(121,\alpha_n)=\frac{1}{24}(n^3-4n)$ if $n$ is even.

First of all suppose that $\sigma$ has $k>2$ blocks.  Since a copy of 121 involves only two blocks at a time, then we know that the partition $\hat{\sigma}$ with same block structure as $\sigma$ arranged in such a way that any two blocks have the structure described in Lemma 4.2 has at least as many copies of 121 as $\sigma$.  

\begin{thm}  For any partition $\pi\in\Pi_n$, $\nu(121,\pi)\leq\nu(121,\alpha_n)$.  \end{thm}

\begin{proof}  Let $g(n)=\left\{\begin{array}{ll}\frac{1}{24}(n^3-n)&n\mbox{ odd,}\\ \frac{1}{24}(n^3-4n)&n\mbox{ even.}\end{array}\right.$  We know that $g(n)$ is the best we can do with at most two blocks in the partition and that this is achieved by $\alpha_n$.  

Suppose that $\sigma\in\Pi_{n,3}$ and has exactly three blocks.  Suppose that there are $a$ 1's, $b$ 2's and $n-a-b$ 3's in the partition $\sigma$.  We know that among partitions with the same block structure as $\sigma$ the one with each pair of blocks arranged as in Lemma 4.2 has the most copies of 121.  Assume that $\sigma$ is arranged in this way.  

Now, the number of copies of 121 involving just the 1's and 2's in this partition is at most $g(a+b)$.  Similarly using the other two pairs of blocks we have at most $g(n-a)$ and $g(n-b)$ copies of 121.  This tells us that the number of copies 121 in this arrangement is bounded by $g(a+b)+g(n-a)+g(n-b)$.  This expression is maximized when $a=b=n/3$.  Thus, the number of copies of 121 is bounded above by $3g(2n/3)\leq\frac{n^3}{27}-\frac{n}{12}$, which is clearly less than $g(n)$.

In general assume that $\sigma\in\Pi_{n,k}$ has exactly $k$ blocks.  Again any two blocks in $\sigma$ when compared to each other must have the arrangement outlined in Lemma 4.2.  By the same argument above the number of copies of 121 in $\sigma$ is bounded above by $\binom{k}{2}g(2n/k)\leq\frac{n^3}{24k}-\frac{n^3}{24k^2}-\frac{n(k-1)}{24}$, which is again less than $g(n)$.  

Thus, $\nu(121,\alpha_n)=\mu(121,n,n).$\end{proof}

Theorem 4.4 tells us that $\delta(121,n,n)=\frac{g(n)}{\binom{n}{3}}$, and thus $\delta(121)=\lim_{n\ra\infty}\frac{g(n)}{\binom{n}{3}}=\frac{1}{4}$.  Notice that this is the first place in which packing densities for set partitions differ from packing densities for words.  It is not a dramatic increase in density, but the unrestricted packing density for 121 is greater than the restricted density for 121 as expected.  This gives us the following results for partitions of [3].  

\bigskip
\begin{center}

\begin{tabular}{|c||c|c|c|c|} \hline
Partition $\pi$&111&112&121&123\\ \hline 
Packing Density $\delta(\pi)$&1&$2\sqrt{3}-3$&$1/4$&1\\ \hline
\end{tabular}
\end{center}

\bigskip

One challenge that the authors found was proving a general result for packing layered set partitions.  For permutations and words it was proved that given a layered permutation pattern or a layered word pattern the object that maximized the number of copies of this pattern was also layered.  Such a proof for set partitions has proved elusive, and is desirable.

\end{document}